\documentclass{article}
\usepackage{graphicx}
\usepackage{cite}
\usepackage{authblk}
\usepackage{amsmath,amssymb,amsthm,mathtools}
\usepackage[left=3cm,right=3cm,top=1.0in,bottom=1.0in]{geometry}
\usepackage[colorlinks=true]{hyperref}
\usepackage[nameinlink,noabbrev]{cleveref}

\newtheorem{theorem}{Theorem}[section]
\newtheorem{lemma}[theorem]{Lemma}

\theoremstyle{remark}
\newtheorem{remark}[theorem]{Remark}

\newcommand{\Gk}{\Gamma_k^+}
\newcommand{\Gtwo}{\Gamma_2^+}
\newcommand{\Sn}{\mathbb{S}}
\newcommand{\Hh}{\mathbb{H}}
\newcommand{\Rr}{\mathbb{R}}

\title{A Model-threshold Dimension Bound and Sharp Critical Ends\\
for Smooth Singular Sets of Constant Positive $\sigma_k$-curvature Metrics}
\author{Jiahuan Li, Yilu Liu, Xi-Nan Ma}
\date{}
\hypersetup{
  pdftitle={A Model-threshold Dimension Bound and Sharp Critical Ends for Smooth Singular Sets of Constant Positive sigma-k Curvature Metrics},
  pdfauthor={Jiahuan Li, Yilu Liu, Xi-Nan Ma}
}

\begin{document}
\maketitle

\begin{abstract}
Let $k\in\mathbb N$ satisfy $1<k<n/2$, and let
$\Sigma^p\subset\Sn^n$ be a closed smooth embedded
submanifold.  We prove that a complete conformal metric
$g=v^{-2}g_{\Sn^n}$ on $\Sn^n\setminus\Sigma$ satisfying
$\lambda(g^{-1}A_g)\in\Gk$ and
$\sigma_k(g^{-1}A_g)=\kappa>0$ must obey
\[
p\leq p_k(n),
\]
where $p_k$ is the model threshold determined by
$\Hh^{p+1}\times\Sn^{n-p-1}$.  When $k=2$ and $n=m^2$, we construct a
smooth complete equality example on
$\Sn^n\setminus\Sn^{(m^2-m-2)/2}$.  We also prove that the strict inequality
$p<p_k(n)$ holds under a finite positive linear-contact hypothesis.
\end{abstract}

\section{Introduction}

For a Riemannian metric $g$ in dimension $n\geq3$, its Schouten tensor is
\[
A_g=\frac{1}{n-2}\left(\operatorname{Ric}_g-
\frac{R_g}{2(n-1)}g\right).
\]
The G\aa rding cone $\Gk$ is the connected component of
$\{\sigma_k>0\}$ containing the positive orthant; equivalently, it consists
of the eigenvalue vectors for which $\sigma_1,\ldots,\sigma_k$ are all
positive~\cite{Garding,Viaclovsky}.  The singular $\sigma_k$-Yamabe problem
asks for complete conformal metrics on $\Sn^n\setminus\Sigma$ whose
Schouten tensor lies in this cone and whose $\sigma_k$ curvature is constant.

Writing $g=e^{-2u}g_{\Sn^n}$, the conformal transformation law is
\[
A_g=A_{g_{\Sn^n}}+\nabla^2u+du\otimes du
-\frac12|\nabla u|^2g_{\Sn^n}.
\]
Thus the curvature prescription is a fully nonlinear Hessian equation.  The
condition $\lambda(g^{-1}A_g)\in\Gk$ selects its elliptic branch, while
completeness imposes a global condition on the degeneration of the conformal
factor along $\Sigma$.  The dimension question considered below comes from
the interaction of these two requirements.

For $k=1$, one has
$\sigma_1(g^{-1}A_g)=R_g/(2(n-1))$, so the problem reduces to the Yamabe
equation.  The compact existence theory was completed by Schoen and is
surveyed by Lee--Parker~\cite{Schoen1984,LeeParker}.  Allowing a singular set
leads to a different program in which completeness replaces smooth extension.
Schoen--Yau obtained the basic Hausdorff-dimension obstruction for complete
locally conformally flat metrics of positive scalar curvature
\cite{SchoenYau}; Mazzeo--Smale constructed metrics on subdomains of the
sphere near equatorial models~\cite{MazzeoSmale}, Mazzeo--Pacard developed a
gluing construction for prescribed higher-dimensional singular sets
\cite{MazzeoPacard1996}, and Korevaar--Mazzeo--Pacard--Schoen established
refined asymptotics in the isolated-singularity setting
\cite{KorevaarMazzeoPacardSchoen}.

For $k\geq2$, the equation is fully nonlinear and is elliptic only after an
admissible branch has been chosen.  The general elliptic theory for symmetric
functions of Hessian eigenvalues was developed by
Caffarelli--Nirenberg--Spruck~\cite{CaffarelliNirenbergSpruck1985}.
In the conformal setting, Viaclovsky placed the $\sigma_k$-curvatures in a
variational framework~\cite{Viaclovsky}, while
Chang--Gursky--Yang treated the four-dimensional $\sigma_2$ equation by
Monge--Amp\`ere methods~\cite{ChangGurskyYang}.  The subsequent compact and
local theories include Schouten-tensor inequalities, a priori estimates,
Liouville and classification results, and existence theorems for higher-order
curvatures~\cite{GuanWang,GuanViaclovskyWang,LiLi2003,LiLi,
ShengTrudingerWang}.
These results provide the elliptic and conformal framework for the admissible
equation.  They do not, by themselves, determine which positive-dimensional
sets can occur as the singular locus of a complete solution.  That question
also depends on the geometry of small tubular neighborhoods of the singular
set.

The singular fully nonlinear theory has developed in parallel.  Singular
radial solutions were classified by Chang--Han--Yang
\cite{ChangHanYang2005}; Han--Li--Teixeira proved precise asymptotics near
isolated singularities~\cite{HanLiTeixeira}; and Mazzieri--Segatti constructed
complete metrics with Delaunay-type ends~\cite{MazzieriSegatti}.

More recent developments concern complete singular metrics and viscosity
solutions of the fully nonlinear Loewner--Nirenberg problem.
Gonz\'alez, Li, and Nguyen proved existence and uniqueness on smooth bounded
domains and obtained a sharp criterion for blow-up along smooth
higher-codimensional boundary components
\cite{GonzalezLiNguyen2018}.  Li and Nguyen subsequently showed that the
viscosity solution on an annulus is locally Lipschitz but fails to be
differentiable across an interior hypersurface \cite{LiNguyen2021}.  Li,
Nguyen, and Xiong established finer boundary regularity for the
$\sigma_k$-Loewner--Nirenberg problem and further non-differentiability
results on multiply connected domains \cite{LiNguyenXiong2023}.  On the
closed sphere, Li, Nguyen, and Wang proved existence and compactness results
for the prescribed $\sigma_k$-curvature problem in the range $k\geq n/2$
\cite{LiNguyenWang2024}.

A complementary weak theory was recently developed by Ma and Wu
\cite{MaWu2025}.  They found a divergence structure for the
$\sigma_k$-Yamabe operator and proved weak continuity of the associated
curvature measure under local $L^1$ convergence.  These works concern
existence, regularity, compactness, and weak convergence.  The question
addressed here is different: we seek a necessary upper bound for the
dimension of a smooth singular set supporting a complete admissible metric
of positive constant $\sigma_k$-curvature.  The smoothness assumption supplies
Fermi coordinates and a controlled splitting into tangential, radial, and
normal-spherical directions.  This makes it possible to compare the equation
at successive tubular scales with the corresponding product model.

For general closed singular sets, the first obstruction relevant here is the
following result of Gonz\'alez.  She proved~\cite[Theorem~1.1]{Gonzalez2005}
that if $g$ is complete,
$\sigma_1\geq c>0$, and $\sigma_2,\ldots,\sigma_k\geq0$, then
\begin{equation}
\dim_H\Sigma\leq\frac{n-2k}{2}.
\label{eq:old-bound}
\end{equation}
The product model $\Hh^{p+1}\times\Sn^{n-p-1}$ predicts a smaller
threshold.  Its Schouten eigenvalues, after multiplication by $2$, are
$+1$ with multiplicity $n-p-1$ and $-1$ with multiplicity $p+1$.  Put
\begin{equation}
c_{n,p,j}:=\sum_{i=0}^j(-1)^{j-i}
\binom{n-p-1}{i}\binom{p+1}{j-i},
\label{eq:c}
\end{equation}
and extend the binomial coefficients polynomially in real $p$.  Define
\begin{equation}
p_k(n):=\sup\{p\geq0:c_{n,p,j}>0\text{ for }1\leq j\leq k\}.
\label{eq:pk}
\end{equation}
For example,
\[
p_2(n)=\frac{n-\sqrt n-2}{2},
\qquad
p_3(n)=\frac{n-2-\sqrt{3n-2}}{2}.
\]
The definition has a direct geometric interpretation.  The complement of an
equatorial $\Sn^p\subset\Sn^n$ is conformal to
$\Hh^{p+1}\times\Sn^{n-p-1}$, and, up to the fixed normalization of the
Schouten eigenvalues, $c_{n,p,j}$ is the $j$-th elementary symmetric function
of the product spectrum.  Hence $p_k(n)$ is the endpoint of model
$\Gk$-admissibility.  In the two displayed cases it improves
\eqref{eq:old-bound} by $(\sqrt n-2)/2$ for $k=2$ and by
$(\sqrt{3n-2}-4)/2$ for $k=3$.
Gonz\'alez--Mazzieri~\cite{GonzalezMazzieri} constructed the corresponding
non-isolated radial models for $k=2$.  Building on the scalar-curvature gluing
scheme, Espinal--Gonz\'alez then produced $\sigma_2$ solutions with prescribed
smooth singular submanifolds below the same threshold
\cite{EspinalGonzalez}.  Their March 2026 revision formulated the conjectured
necessary condition for general $k$ as $p<p_k(n)$.

Our results describe both the force and the limitation of this model
prediction.  For every $1<k<n/2$, we prove the model-threshold bound without
assuming an asymptotic expansion of the conformal factor.  A finite positive
linear-contact condition excludes the endpoint.  In contrast, for $k=2$ we
construct complete critical metrics at the integral model endpoints, showing
that the non-strict conclusion is optimal there.  We first state the general
dimension theorem.

\begin{theorem}\label{thm:main}
Let $k\in\mathbb N$ satisfy $1<k<n/2$, let
$\Sigma^p\subset\Sn^n$ be a closed smooth embedded
submanifold, and let
\[
g=v^{-2}g_{\Sn^n}
\]
be complete on $\Sn^n\setminus\Sigma$.  Assume
\[
\lambda(g^{-1}A_g)\in\Gk,
\qquad
\sigma_k(g^{-1}A_g)\equiv\kappa>0.
\]
Then
\[
p\leq p_k(n).
\]
\end{theorem}

The theorem converts the algebraic failure of the product spectrum to lie in
$\Gk$ into an obstruction to completeness.  Since no limit for
$v/\operatorname{dist}(\cdot,\Sigma)$ is assumed, the argument uses upper
contacts rather than a prescribed leading asymptotic term.  We briefly
indicate the main mechanism.  Chang--Han--Yang ball
convexity~\cite[Theorem~1.10]{ChangHanYang} gives the scale-invariant estimate
\[
\rho |D\log v|\leq C,
\qquad
\rho=\operatorname{dist}(\,\cdot\,,\Sigma).
\]
The Fermi-coordinate calculation, combined with the supercritical separation
from the model G{\aa}rding cone, shows that every admissible upper test $\phi$
of the radial logarithmic profile satisfies
\[
\phi'\leq\bar a
\quad\Longrightarrow\quad
\phi''\geq h
\]
for uniform constants $\bar a,h>0$.  A quadratic tangential penalty and a
one-dimensional propagation estimate then produce a fixed increase on each
logarithmic tubular scale.  The resulting connecting curves have summable
$g$-lengths, contradicting completeness.  For $k=2$, the cone separation is
an explicit quadratic calculation; we record this specialization after the
general proof for comparison.

For $k=2$, we construct explicitly the case that $p=p_2(n)$.

\begin{theorem}\label{thm:sharp}
For every integer $m\geq3$ and every $\kappa>0$, put
\[
n=m^2,
\qquad
p=\frac{m^2-m-2}{2}=p_2(n).
\]
There exists a smooth complete metric $g$ on
$\Sn^n\setminus\Sn^p$ such that
\[
g\in[g_{\Sn^n}],
\qquad
\lambda(g^{-1}A_g)\in\Gtwo,
\qquad
\sigma_2(g^{-1}A_g)\equiv\kappa.
\]
\end{theorem}

The construction starts from the critical product
$\Hh^{p+1}\times\Sn^{n-p-1}$, whose $\sigma_2$ curvature vanishes at the
endpoint, and deforms it within the spherical conformal class to obtain any
prescribed constant $\kappa>0$ while retaining completeness and
$\Gtwo$-admissibility.  The smallest example is a complete metric on
$\Sn^9\setminus\Sn^2$.  Although the equality examples rule out a universal
strict theorem, we prove that strictness is recovered under the finite
positive linear-contact condition
\[
0<\limsup_{x\to\Sigma}
\frac{v(x)}{\operatorname{dist}(x,\Sigma)}<\infty.
\]
This condition describes the regime in which the conformal factor has a
nondegenerate linear scale along a sequence approaching $\Sigma$.  The
rescaled Schouten tensors then approach the product spectrum strongly enough
for the positive constant $\sigma_k$ equation to exclude equality in the
model coefficient.

Section~2 develops the model algebra, cylindrical estimates, Fermi-coordinate
upper-contact calculation, and regular-singular lemma.  Section~3 proves the
dimension bound for arbitrary $k$ and the strict inequality under finite
positive linear contact.  Section~4 gives the explicit quadratic
specialization, constructs the global equality examples, and states the
resulting optimality at integral model endpoints.

\section{Preliminaries and Some Basic Facts}

Throughout this section, we assume that
$(g,\Sigma,k,n,p,\kappa)$ satisfy the hypotheses of \cref{thm:main}, unless otherwise specified.

\subsection{Algebraic facts and cylindrical estimates}

Let
\begin{equation}
\mathcal P_{j,n}(d):=j![t^j](1+t)^{(n+d)/2}(1-t)^{(n-d)/2}.
\label{eq:P}
\end{equation}
Differentiating the generating function gives that
\begin{equation}
\mathcal P_{0,n}=1,
\quad
\mathcal P_{1,n}=d,
\quad
\mathcal P_{j+1,n}=d\mathcal P_{j,n}
-j(n-j+1)\mathcal P_{j-1,n}.
\label{eq:recurrence}
\end{equation}

For $q=n-p-1$ and $\alpha\in\Rr$, set
\[
D(\alpha):=
\bigl((1+\alpha)^{[q]},-(1+\alpha),
(-(1-\alpha))^{[p]}\bigr),
\qquad
R_j(\alpha):=\sigma_j(D(\alpha)).
\]
\begin{lemma}\label{lem:model}
Let $d_{k,n}$ be the largest zero of $\mathcal P_{k,n}$.  Then
\[
p_k(n)=\frac{n-2-d_{k,n}}{2}.
\]
For $d\geq0$,
\begin{align}\label{equal1}
    \mathcal P_{j,n}(d)>0\quad(1\leq j\leq k)
\quad\Longleftrightarrow\quad d>d_{k,n},
\end{align}

and
\begin{align}\label{equal2}
    \mathcal P_{j,n}(d)\geq0\quad(1\leq j\leq k)
\quad\Longleftrightarrow\quad d\geq d_{k,n}.
\end{align}

Moreover, if $p>p_k(n)$, then there exist
$\bar a\in(0,1)$ and $\delta_R>0$ such that
\[
\min_{1\leq j\leq k}R_j(\alpha)\leq-\delta_R
\]
for every $\alpha\leq\bar a$.
\end{lemma}

\begin{proof}
We obeserve that\eqref{eq:recurrence} is a monic three-term
Jacobi recurrence whose recurrence coefficients
$j(n-j+1)$ are positive for $j<n$.  Standard Jacobi-matrix
theory therefore implies that every zero of
$\mathcal P_{j,n}$ is real and simple and that the zeros of
consecutive polynomials strictly interlace \cite{Szego}.

Let $r_j$ denote the largest zero of $\mathcal P_{j,n}$, then
\[
0=r_1<r_2<\cdots<r_k.
\]
If $d\geq r_k$, then $d\geq r_j$ for every $1\leq j\leq k$.
Since $\mathcal P_{j,n}$ is monic and has no zero larger than
$r_j$, it follows that
\[
\mathcal P_{j,n}(d)\geq0
\qquad (1\leq j\leq k).
\]
All these inequalities are strict when $d>r_k$.

Conversely, suppose that $0\leq d<r_k$, and define
\[
m:=\min\{\ell\in\{2,\ldots,k\}:d<r_\ell\}.
\]
The minimality of $m$ gives
\[
r_{m-1}\leq d<r_m.
\]
By strict interlacing, the second largest zero of
$\mathcal P_{m,n}$ lies strictly below $r_{m-1}$.
Consequently,
\[
\mathcal P_{m,n}(d)<0
\qquad\text{if }r_{m-1}<d<r_m.
\]
For $d=r_{m-1}$, the recurrence
\eqref{eq:recurrence} and
$\mathcal P_{m-1,n}(r_{m-1})=0$ give
\[
\mathcal P_{m,n}(r_{m-1})
=-(m-1)(n-m+2)
  \mathcal P_{m-2,n}(r_{m-1})<0.
\]
For $m\geq3$, we have $r_{m-1}>r_{m-2}$, and hence
\[
\mathcal P_{m-2,n}(r_{m-1})>0.
\]
The same conclusion holds for $m=2$ because
$\mathcal P_{0,n}=1$.  Therefore
$\mathcal P_{m,n}(d)<0$ throughout the interval
\[
r_{m-1}\leq d<r_m.
\]
It follows that the polynomials
$\mathcal P_{1,n}(d),\ldots,\mathcal P_{k,n}(d)$ cannot all
be nonnegative when $0\leq d<r_k$.  Since
$d_{k,n}=r_k$, we derive \eqref{equal1} and \eqref{equal2}.

The preceding argument proves both sign characterizations.

To compute the threshold, set
\[
d=n-2p-2.
\]
Then
\[
\frac{n+d}{2}=n-p-1=q,
\qquad
\frac{n-d}{2}=p+1,
\]
and \eqref{eq:P} gives
\[
\mathcal P_{j,n}(d)
=j![t^j](1+t)^q(1-t)^{p+1}
=j!c_{n,p,j}.
\]
Since $\mathcal P_{1,n}(d)=d$, simultaneous positivity of
$c_{n,p,1},\ldots,c_{n,p,k}$ requires $d>0$.  The strict
sign characterization proved above therefore gives
\[
c_{n,p,j}>0\quad(1\leq j\leq k)
\quad\Longleftrightarrow\quad
n-2p-2>d_{k,n}.
\]
Equivalently,
\[
c_{n,p,j}>0\quad(1\leq j\leq k)
\quad\Longleftrightarrow\quad
p<\frac{n-2-d_{k,n}}{2}.
\]
Taking the supremum in the definition of $p_k(n)$ yields
\[
p_k(n)=\frac{n-2-d_{k,n}}{2}.
\]

We now prove the uniform separation in the last part of the
lemma.  By \eqref{eq:old-bound},
\[
p\leq\frac{n-2k}{2},
\]
and hence
\[
d=n-2p-2\geq2k-2>0.
\]
Suppose that $p>p_k(n)$.  The threshold formula gives
$d<d_{k,n}$.  The nonnegative sign characterization then
provides an index $j_0\in\{1,\ldots,k\}$ such that
\[
\mathcal P_{j_0,n}(d)<0.
\]
Since
\[
R_j(0)=\sigma_j(D(0))
=c_{n,p,j}
=\frac{1}{j!}\mathcal P_{j,n}(d),
\]
we conclude that
\[
D(0)\notin\overline{\Gk}.
\]

For $-1<\alpha\leq0$, define
\[
\widehat D(\alpha):=\frac{D(\alpha)}{1+\alpha}.
\]
The definition of $D(\alpha)$ gives
\[
D(0)-\widehat D(\alpha)
=
\left(
0^{[q+1]},
\left(-\frac{2\alpha}{1+\alpha}\right)^{[p]}
\right).
\]
Every component on the right-hand side is nonnegative when
$-1<\alpha\leq0$, and therefore
\[
D(0)-\widehat D(\alpha)\in\mathbb R_+^n.
\]

The closed G\aa rding cone is monotone under addition of a
nonnegative vector \cite{Garding,Viaclovsky}:
\[
\overline{\Gk}+\mathbb R_+^n
\subset\overline{\Gk}.
\]
If $\widehat D(\alpha)$ belonged to $\overline{\Gk}$, then
\[
D(0)
=\widehat D(\alpha)
 +\bigl(D(0)-\widehat D(\alpha)\bigr)
\]
would also belong to $\overline{\Gk}$, contradicting the
exclusion of $D(0)$.  Thus
\[
\widehat D(\alpha)\notin\overline{\Gk}
\qquad(-1<\alpha\leq0).
\]
Since $D(\alpha)=(1+\alpha)\widehat D(\alpha)$ and
$1+\alpha>0$, positive homogeneity of the closed cone gives
\[
D(\alpha)\notin\overline{\Gk}
\qquad(-1<\alpha\leq0).
\]
Consequently,
\[
\min_{1\leq j\leq k}R_j(\alpha)<0
\qquad(-1<\alpha\leq0).
\]

Because $\overline{\Gk}$ is closed and $D(\alpha)$ depends
continuously on $\alpha$, there exists $\bar a\in(0,1)$ such
that
\[
D(\alpha)\notin\overline{\Gk}
\qquad(0\leq\alpha\leq\bar a).
\]
For $\alpha\leq-1$, we have
\[
R_1(\alpha)
=d+(n-2)\alpha
\leq d-(n-2)
=-2p.
\]
Moreover, $p$ is an integer and $p>p_k(n)\geq0$, so $p\geq1$
and the right-hand side is strictly negative.

Define
\[
\mu(\alpha):=\min_{1\leq j\leq k}R_j(\alpha).
\]
The preceding conclusions show that $\mu(\alpha)<0$ on the
compact interval $[-1,\bar a]$.  Therefore
\[
\delta_0:=
-\max_{\alpha\in[-1,\bar a]}\mu(\alpha)>0.
\]
For $\alpha\leq-1$, the estimate for $R_1$ gives
\[
\mu(\alpha)\leq R_1(\alpha)\leq-2p.
\]
Hence the required uniform estimate holds with
\[
\delta_R:=\min\{\delta_0,2p\}>0.
\]
\end{proof}

We combine this algebraic separation with estimates on tubular scales.
Choose a stereographic pole in the regular domain so that $\Sigma$ is
compact in a Euclidean chart and write $g=v^{-2}|dx|^2$.  In these
coordinates
\begin{equation}
\mathcal A[v]:=g^{-1}A_g=vD^2v-\frac12|Dv|^2I.
\label{eq:A-v}
\end{equation}
The Newton--Maclaurin inequalities on $\Gk$
\cite{Garding,GuanViaclovskyWang} give
\[
\sigma_1(\mathcal A[v])
\geq n\left(\frac{\kappa}{\binom nk}\right)^{1/k}>0.
\]
Since
\[
R_g=2(n-1)\sigma_1(g^{-1}A_g),
\]
the scalar curvature has a uniform positive lower bound.  The boundary lemma
of Gonz\'alez~\cite[Lemma 3.1]{Gonzalez2005} therefore applies and shows that
the spherical conformal factor tends uniformly to zero at $\Sigma$.  To
compare the two gauges, write
\[
\Pi^*g_{\Sn^n}=\omega^2|dx|^2.
\]
If the spherical conformal factor is $v_0$, then the Euclidean factor used
here is
\[
v=\frac{v_0\circ\Pi}{\omega}.
\]
The image of $\Sigma$ is compact in the chosen Euclidean chart, so $\omega$
is bounded above and below by positive constants near $\Sigma$.  Hence
$v\to0$ uniformly at $\Sigma$ as well.  Finally,
$\lambda(\mathcal A[v])\in\Gk$ and $\sigma_k(\mathcal A[v])=\kappa>0$
place the equation on the fixed elliptic branch required by Chang--Han--Yang
ball convexity~\cite[Theorem 1.10]{ChangHanYang}.

\begin{lemma}\label{lem:scale-gradient}
For sufficiently small $\rho=\operatorname{dist}(x,\Sigma)$,
\begin{equation}
\rho|D\log v(x)|\leq2.
\label{eq:scale-gradient}
\end{equation}
Consequently, for $t=-\log r$ and $f=\log(v/r)$,
\[
|\partial_t f|\leq L_0
\]
with a constant independent of the tubular scale.
\end{lemma}

\begin{proof}
Ball convexity applies to every Euclidean ball compactly contained in the
domain.  On $\partial B_R(c)$ it gives $\nu(\log v)<1/R$.  Given $x$, a unit vector
$e$, and $R<\rho(x)/2$, apply this to the two balls centered at
$x\mp Re$.  Letting $R\uparrow\rho/2$ gives
\[
|e\cdot D\log v|\leq2/\rho.
\]
Taking the supremum over all unit vectors $e$ gives
\[
\rho |D\log v|\leq 2,
\]
and hence the stated estimate.  Since $r=\rho$ in the Euclidean tubular neighborhood and $\partial_t=-r\partial_r$, we also have
\[
|\partial_t\log v|\leq2.
\]
Because $f=\log v+t$, it follows that
\[
|\partial_t f|\leq3.
\]
\end{proof}

The metric is complete and locally conformally flat.  Moreover,
$\lambda(\mathcal A[v])\in\Gamma_k^+$ gives
$\sigma_j(\mathcal A[v])>0$ for $1\leq j\leq k$, and the
Newton--Maclaurin inequality above gives the uniform lower bound
\[
\sigma_1(\mathcal A[v])
\geq n\left(\frac{\kappa}{\binom nk}\right)^{1/k}>0.
\]
Thus all the hypotheses of Gonz\'alez's dimension estimate
\cite[Theorem~1.1]{Gonzalez2005} are satisfied.  Since $\Sigma$ is a smooth
$p$-dimensional submanifold, $\dim_H\Sigma=p$, and therefore
\begin{equation}
p\leq\frac{n-2k}{2},
\qquad
d=n-2p-2\geq2k-2>0.
\label{eq:old-use}
\end{equation}

\subsection{Fermi-coordinate upper contacts and iteration}

Put $q=n-p-1$ and use Fermi variables $(r,\theta,y)$ near $\Sigma$.
For an upper test
\begin{equation}
\Phi=r\exp\{\eta(-\log r)+P(y)\},
\label{eq:test}
\end{equation}
set
\[
M:=2(\Phi/r)^{-2}\mathcal A[\Phi].
\]
In the flat product geometry, direct calculation gives
\begin{equation}
M_0=
\begin{pmatrix}
2\eta''+\varrho&\beta^T\\
\beta&C
\end{pmatrix},
\label{eq:M0}
\end{equation}
where
\begin{align*}
\varrho&=(\eta')^2-1-r^2|D_\Sigma P|^2,\\
\beta&=\bigl(0_q,2(1-\eta')rD_\Sigma P\bigr)\in\mathbb R^{q+p},\\
C&=
\begin{pmatrix}
\bigl(1-(\eta')^2-r^2|D_\Sigma P|^2\bigr)I_q&0\\
0&2r^2D_\Sigma P\otimes D_\Sigma P
-\bigl((1-\eta')^2+r^2|D_\Sigma P|^2\bigr)I_p
\end{pmatrix}.
\end{align*}

For a symmetric matrix $B$, we use the Newton tensors
\[
T_m(B):=\sum_{\ell=0}^{m}(-1)^\ell
\sigma_{m-\ell}(B)B^\ell,
\qquad
T_{-1}(B):=0.
\]
Thus $T_m$ is the derivative tensor associated with $\sigma_{m+1}$.

For $1\leq j\leq k$,
\[
\sigma_j(M_0)=2\eta''\sigma_{j-1}(C)+G_j,
\]
where
\[
G_j=\varrho\sigma_{j-1}(C)+\sigma_j(C)
-\beta^TT_{j-2}(C)\beta,
\]
with $T_{-1}=0$.  If $M_0\in\Gk$, then $C\in\Gamma_{k-1}^+$.

Indeed, expansion by principal minors gives both formulas.  For
$1\leq m\leq k-1$,
\[
\sigma_m(C)=\langle T_m(M_0)e_r,e_r\rangle>0,
\]
because $T_m(M_0)$ is positive definite on $\Gk$ for $m<k$; see
\cite{Garding,GuanViaclovskyWang,Viaclovsky}.  Thus
$C\in\Gamma_{k-1}^+$.

\begin{lemma}\label{lem:fermi-block}
Fix $H_*>0$.  There is a sufficiently small tubular neighborhood such that,
whenever
\[
|\eta'|+r|D_\Sigma P|\leq H_*,
\qquad
r^2|D_\Sigma^2P|\leq1,
\]
the matrix $M$ admits the exact decomposition
\begin{equation}
M=
\begin{pmatrix}
2\eta''+\widetilde\varrho&\widetilde\beta^T\\
\widetilde\beta&\widetilde C
\end{pmatrix},
\qquad
\sigma_j(M)=2\eta''\sigma_{j-1}(\widetilde C)+\widetilde G_j,
\label{eq:radial-block}
\end{equation}
where
\[
\widetilde G_j
=\widetilde\varrho\sigma_{j-1}(\widetilde C)
+\sigma_j(\widetilde C)
-\widetilde\beta^TT_{j-2}(\widetilde C)\widetilde\beta
\]
is independent of $\eta''$.  Uniformly for $1\leq j\leq k$,
\begin{equation}
\|\widetilde C-C\|+|\widetilde\varrho-\varrho|
+|\widetilde\beta-\beta|+|\widetilde G_j-G_j|
\leq C_{H_*}\bigl(r+r^2|D_\Sigma^2P|\bigr).
\label{eq:tube-perturbation}
\end{equation}
If $M\in\Gk$, then $\widetilde C\in\Gamma_{k-1}^+$.
\end{lemma}

\begin{proof}
Put $\ell=\log\Phi$.  In Euclidean coordinates,
\[
M=2r^2\left(D^2\ell+d\ell\otimes d\ell
-\frac12|d\ell|^2I\right).
\]
If $\pi$ is nearest-point projection onto $\Sigma$, then
\[
\ell=\log r+\eta(-\log r)+P\circ\pi,
\qquad
D^2(\eta(-\log r))=\eta''\,dt\otimes dt+\eta'D^2t.
\]
Since
\[
r^2dt\otimes dt=dr\otimes dr,
\]
the term $\eta''dt\otimes dt$ contributes
$2\eta''dr\otimes dr$ to $M$.  The standard tube
identities~\cite{GrayTubes} are
\[
rD^2r=\Pi_{\rm ang}+O(r),
\qquad
  D\pi=\Pi_{T\Sigma}+O(r),
\qquad
  |D^2\pi|\leq C
\]
and the chain rule give
  \[
  D^2(P\circ\pi)
  =(D\pi)^*(D_\Sigma^2P)
  +\langle D_\Sigma P,D^2\pi\rangle.
  \]
  Using $D\pi=\Pi_{T\Sigma}+O(r)$, we obtain
  \[
  r^2D^2(P\circ\pi)
  =r^2D_\Sigma^2P
  +O\left(r^3|D_\Sigma^2P|+r^2|D_\Sigma P|\right).
  \]
  Under $r|D_\Sigma P|\leq H_*$ and
  $r^2|D_\Sigma^2P|\leq1$, the additional error is
  $O_{H_*}(r)$.
Substitution gives
\[
M=M_0+E,
\]
where the symmetric error $E$ is independent of $\eta''$ and satisfies
\[
\|E\|\leq C_{H_*}\bigl(r+r^2|D_\Sigma^2P|\bigr).
\]
Writing the radial and complementary blocks of $E$ separately gives
\[
\|\widetilde C-C\|+|\widetilde\varrho-\varrho|
+|\widetilde\beta-\beta|
\leq C_{H_*}\bigl(r+r^2|D_\Sigma^2P|\bigr).
\]
Expansion of the principal minors according to whether they contain the
radial direction gives \eqref{eq:radial-block}.  Since
\[
G_j(\varrho,\beta,C)
=\varrho\sigma_{j-1}(C)+\sigma_j(C)
-\beta^TT_{j-2}(C)\beta
\]
is a polynomial in the block variables, its first derivatives are uniformly
bounded under the stated contact bounds.  The mean-value theorem therefore
gives
\[
|\widetilde G_j-G_j|
\leq C_{H_*}\bigl(r+r^2|D_\Sigma^2P|\bigr),
\qquad 1\leq j\leq k.
\]
This proves \eqref{eq:tube-perturbation}.  Finally, for
$0\leq m\leq k-1$,
\[
\sigma_m(\widetilde C)=\langle T_m(M)e_r,e_r\rangle>0,
\]
by positivity of the Newton tensors in $\Gamma_k^+$.  Hence
$\widetilde C\in\Gamma_{k-1}^+$.
\end{proof}

The spectrum of $C$ consists of
\begin{align*}
&1-(\eta')^2-r^2|D_\Sigma P|^2
&&\text{with multiplicity }q,\\
&-\bigl((1-\eta')^2+r^2|D_\Sigma P|^2\bigr)
&&\text{with multiplicity }p-1,\\
&r^2|D_\Sigma P|^2-(1-\eta')^2
&&\text{with multiplicity }1.
\end{align*}
Whenever $\eta'<1$, set
\[
\chi=\frac{r^2|D_\Sigma P|^2}{1-\eta'},
\qquad
\alpha=\eta'-\chi,
\]
and define
\[
H_m(\alpha):=[s^m]
(1+(1+\alpha)s)^q(1-(1-\alpha)s)^{p-1},
\]
with $H_{-1}=0$.  Expansion using the spectrum after deleting the last
direction gives
\begin{equation}
G_j=(1-\eta')^j
\{R_j(\alpha)-4\chi H_{j-2}(\alpha)\}.
\label{eq:Gidentity}
\end{equation}

\begin{lemma}\label{lem:acceleration}
Under the standing hypotheses, assume $p>p_k(n)$.  There exist
$\bar a,h,r_0>0$ and $\varepsilon_0\in(0,1)$ such that the following holds.  If an upper
 test \eqref{eq:test} at $r<r_0$ satisfies
\[
|\eta'|+r|D_\Sigma P|\leq H_*,
\qquad
\eta'\leq\bar a,
\qquad
r^2|D^2_\Sigma P|\leq\varepsilon_0,
\qquad
\lambda(\mathcal A[\Phi])\in\Gk,
\]
then $\eta''\geq h$.
\end{lemma}

\begin{proof}
First suppose the tube and test are flat.  Delete from $C$ the eigenvalue in
the $z$ direction.  The remaining spectrum is
  \[
  \bigl((1-\eta')(1+\alpha)\bigr)^{[q]},
  \qquad
  \bigl(-(1-\eta')(1-\alpha)\bigr)^{[p-1]}.
  \]
Since $C\in\Gamma_{k-1}^+$, positivity of its Newton tensors gives
$H_m(\alpha)>0$ for $0\leq m\leq k-2$.  Hence
\[
  G_j\leq(1-\eta')^jR_j(\alpha).
  \]
  Moreover, $\alpha\leq\eta'\leq\bar a$.  By the supercritical
  separation part of \cref{lem:model}, some $j\leq k$ has
  $R_j(\alpha)\leq-\delta_R$.  Since $1-\eta'\geq1-\bar a$, it follows that
  
  \[
  G_j\leq-\delta_R
  \min_{1\leq\ell\leq k}(1-\bar a)^\ell<0.
  \]

For a curved tube, the estimate can be transferred without assuming that the
flat complementary block itself stays in the cone.  Put
\[
\epsilon=r+r^2|D_\Sigma^2P|,
\]
  and let $e$ be the unit vector in the $D_\Sigma P$ direction; when
  $D_\Sigma P=0$, choose any unit vector in the tangential $y$-block.
  Compression to $e^\perp$ and the principal
minor identity give
\begin{equation}
\left\langle T_m(C)e,e\right\rangle
  =(1-\eta')^mH_m(\alpha),
\qquad 0\leq m\leq k-2.
\label{eq:H-newton}
\end{equation}
The exact block formula applied to $M\in\Gk$ gives
$\widetilde C\in\Gamma_{k-1}^+$, and hence
\[
\left\langle T_m(\widetilde C)e,e\right\rangle>0.
\]
By the tube-perturbation estimate and the polynomial mean-value estimate used
in its proof,
\[
  (1-\eta')^mH_m(\alpha)\geq-C_0\epsilon
\qquad 0\leq m\leq k-2.
\]
  Here $1-\eta'\geq1-\bar a>0$ and
  \[
  0\leq\chi
  =\frac{r^2|D_\Sigma P|^2}{1-\eta'}
  \leq\frac{H_*^2}{1-\bar a}.
  \]
  Thus
\eqref{eq:Gidentity}, including the case $j=1$ where $H_{-1}=0$, implies
\[
  G_j\leq(1-\eta')^jR_j(\alpha)+C_1\epsilon.
\]
For the index selected by the supercritical separation estimate, put
\[
\delta_0=\delta_R\min_{1\leq \ell\leq k}(1-\bar a)^\ell.
\]
After decreasing $r_0$ and $\varepsilon_0$ so that
\[
(C_1+C_{H_*})\epsilon\leq\frac{\delta_0}{2},
\]
and using $|\widetilde G_j-G_j|\leq C_{H_*}\epsilon$, we obtain
\[
0<\sigma_j(M)
  =2\eta''\sigma_{j-1}(\widetilde C)+\widetilde G_j,
\qquad
\widetilde G_j\leq-\frac{\delta_0}{2}.
\]
  The contact bounds and \eqref{eq:tube-perturbation} imply
  \[
  \|\widetilde C\|\leq C(H_*).
  \]
  Since each $\sigma_m$ is a polynomial in the matrix coefficients and
  $0\leq m\leq k-1$, there is
  $S_{\mathrm{up}}=S_{\mathrm{up}}(H_*,n,k)>0$ such that
  \[
  0<\sigma_m(\widetilde C)\leq S_{\mathrm{up}},
  \qquad 0\leq m\leq k-1.
  \]
  Therefore
  $\eta''\geq\delta_0/(4S_{\mathrm{up}})=:h$.
\end{proof}

\begin{lemma}\label{lem:peak-preparations}
\begin{enumerate}
\item[(i)] Let $L_0\geq0$, $\bar a,h>0$, and let
$u:[0,T]\to\Rr$ be $L_0$-Lipschitz and locally semiconvex.  Suppose every
$C^2$ upper test $\phi$ at an interior point satisfies
\[
\phi'\leq\bar a\quad\Longrightarrow\quad\phi''\geq h.
\]
Then, for every $0<c<\bar a$,
\[
u(T)-u(0)\geq cT-\frac{(L_0+c)^2}{2h}.
\]

\item[(ii)] Assume $\operatorname{codim}\Sigma\geq2$.  There are
$r_*>0$ and $C_*>0$, depending only on the fixed tubular neighborhood and
the constant in \eqref{eq:scale-gradient}, with the following properties.
Fix $T,\Lambda\geq1$, let $0<R<r_*$, put $R'=e^{-T}R$, and assume
$\Lambda R<\operatorname{inj}(\Sigma)/4$.  On the Fermi box
\[
R'\leq r\leq R,
\qquad
d_\Sigma(y,y_0)\leq\Lambda R,
\]
one has
\begin{equation}
\operatorname{osc}_{\rm box} f
\leq C_*\bigl(\Lambda e^T+T+1\bigr).
\label{eq:box-osc}
\end{equation}
If
\[
P(y)=\mu\frac{d_\Sigma(y,y_0)^2}{R^2},
\]
then throughout the box
\begin{equation}
r|D_\Sigma P|\leq C_*\mu\Lambda,
\qquad
r^2|D_\Sigma^2P|\leq C_*\mu.
\label{eq:box-penalty}
\end{equation}
Finally, if $X=(y,R,\theta)$ and $X'=(y',R',\theta')$ are points of the
box with $d_\Sigma(y,y')\leq\Lambda R$, there is a Fermi coordinate curve
$\gamma$ from $X$ to $X'$ which stays where $r\geq R'$ and satisfies
\begin{equation}
\operatorname{Length}_{\rm Eucl}(\gamma)\leq C_*(\Lambda+1)R,
\qquad
\operatorname{Length}_g(\gamma)
\leq C_{T,\Lambda}\frac{R}{R'e^{f(X')}}.
\label{eq:box-path}
\end{equation}
\end{enumerate}
\end{lemma}

\begin{proof}
We first prove (i).  At an Alexandrov point where $u'<c$, adding
$\varepsilon(t-t_0)^2$ to the second-order Taylor polynomial gives a local
$C^2$ upper test.  Letting $\varepsilon\downarrow0$ in the assumed test
inequality yields
\[
u''\geq h
\quad\text{almost everywhere on }\{u'<c\}.
\]
Let $\mathfrak a$ be the precise BV representative of $u'$.  Semiconvexity
implies that the singular part of $D\mathfrak a$ is nonnegative.  The BV
chain rule applied to $q_0=(c-\mathfrak a)_+$ therefore gives
\cite[Theorem~3.99]{AmbrosioFuscoPallara}
\[
Dq_0\leq-h\mathbf 1_{\{q_0>0\}}\,dt.
\]
Comparison with the solution of $y'=-h$ on $\{y>0\}$, together with
$q_0(0)\leq L_0+c$, yields
\[
q_0(t)\leq(L_0+c-ht)_+
\]
for almost every $t\in(0,T)$.  Consequently
\[
u'(t)\geq\min\{-L_0+ht,c\}
\]
almost everywhere.  Integration gives
\[
u(T)-u(0)
\geq\int_0^T\min\{-L_0+ht,c\}\,dt
\geq cT-\frac{(L_0+c)^2}{2h},
\]
which proves (i).

We next prove (ii).  The non-polar tubular map~\cite{GrayTubes}
\[
(y,\xi)\longmapsto y+\xi,
\qquad
\xi\in N_y\Sigma,
\]
and its inverse have uniformly bounded first two derivatives on the fixed
tubular neighborhood.  After writing $\xi=r\theta$, the angular derivatives
carry the standard $r$-scaling.  Join two points in the box by a radial
segment, two minimizing base geodesics from the first base point to $y_0$
and from $y_0$ to the second base point, with normal parallel transport on
both pieces, and an angular segment in the normal sphere.  This broken base
path stays in the stated base ball and has length at most $2\Lambda R$.
Along the resulting curve the scale-gradient estimate and
$|D\log r|=1/r$ give \eqref{eq:box-osc}: the radial, horizontal, and angular
parts contribute respectively $O(T)$, $O(\Lambda R/R')$, and $O(1)$.
The injectivity-radius choice makes $d_\Sigma^2$ smooth on the relevant ball,
where
\[
|D_\Sigma d_\Sigma^2|\leq C\Lambda R,
\qquad
|D_\Sigma^2d_\Sigma^2|\leq C.
\]
This proves \eqref{eq:box-penalty}.  The same piecewise Fermi curve has the
Euclidean-length bound in \eqref{eq:box-path}.  Moreover,
\eqref{eq:box-osc} gives
\[
v\geq e^{-C_{T,\Lambda}}v(X')
=e^{-C_{T,\Lambda}}R'e^{f(X')}
\]
on the curve.  Integrating $|dx|/v$ proves the metric-length estimate.
\end{proof}

\subsection{Regular-singular smoothness at the collapsed orbit}

\begin{lemma}\label{lem:regular-singular}
Let $\mathcal F$ and $\mathcal G$ be smooth near
$(0,A_\circ,Z_\circ)$, and suppose
\[
A,Z\in C([0,x_0])\cap C^\infty((0,x_0])
\]
satisfy
\begin{align}
A+2xA_x&=\mathcal F(x,A,Z),
\label{eq:regular-singular-A}\\
Z_x&=\mathcal G(x,A,Z)
\label{eq:regular-singular-Z}
\end{align}
on $(0,x_0]$.  Assume
\[
A_\circ=\mathcal F(0,A_\circ,Z_\circ),
\qquad
A(0)=A_\circ,
\qquad
Z(0)=Z_\circ,
\qquad
A(x)-A_\circ=O(x),
\]
and, for every integer $\ell\geq1$,
\[
2\ell+1-\partial_A\mathcal F(0,A_\circ,Z_\circ)>0.
\]
Then $A$ and $Z$ extend smoothly to $x=0$.
\end{lemma}

\begin{proof}
The induction uses the following scalar estimate.  Suppose $Y$ is
bounded near zero and
\[
2xY'+\lambda(x)Y=H(x),
\qquad
\lambda(0)>0.
\]
The bounded solution is uniquely selected by
\[
Y(x)=\frac12\int_0^x
\exp\left(-\frac12\int_t^x\frac{\lambda(s)}s\,ds\right)
\frac{H(t)}t\,dt.
\]
After the substitution $t=xr$, for sufficiently small $x$ and every
$0\leq\ell\leq j$, the $\ell$th $x$-derivative of the integrand is bounded by
\[
C_\ell r^{\lambda(0)/4-1}
\bigl(1+|\log r|^{N_\ell}\bigr),
\qquad 0<r<1,
\]
for some integer $N_\ell\geq0$.  Since $\lambda(0)>0$,
\[
\int_0^1r^{\lambda(0)/4-1}
\bigl(1+|\log r|^{N_\ell}\bigr)\,dr<\infty.
\]
Differentiation under the integral sign shows that if $\lambda$ and $H$ are
$C^j$ at zero, then $Y$ is $C^j$ there and
\[
Y(0)=\frac{H(0)}{\lambda(0)}.
\]

Equation~\eqref{eq:regular-singular-Z} gives $Z\in C^1$ and
$Z-Z_\circ=O(x)$, because its right-hand side extends continuously to zero.
The compatibility relation and $A-A_\circ=O(x)$, together with
\eqref{eq:regular-singular-A}, imply that $A'$ is bounded on $(0,x_0]$.
Differentiating \eqref{eq:regular-singular-A} for $x>0$ and putting $Y=A'$
gives
\[
2xY'+\bigl(3-\partial_A\mathcal F(x,A,Z)\bigr)Y
=\partial_x\mathcal F(x,A,Z)
+\partial_Z\mathcal F(x,A,Z)Z'.
\]
The coefficient and right-hand side extend continuously to zero, and the
limiting coefficient is positive.  The scalar estimate therefore extends
$A'$ continuously to zero.  Thus $A\in C^1$, and the regular equation for
$Z$ now gives $Z\in C^2$.

Inductively, assume $A\in C^\ell$ and $Z\in C^{\ell+1}$.  Differentiating
\eqref{eq:regular-singular-A} $\ell$ times and isolating the only term containing
$A^{(\ell)}$ from the chain rule gives, for $Y=A^{(\ell)}$,
\[
2xY'+
\bigl(2\ell+1-\partial_A\mathcal F(x,A,Z)\bigr)Y=H_\ell(x).
\]
Here $H_\ell$ contains only derivatives of $A$ through order $\ell-1$ and
of $Z$ through order $\ell$, so it is $C^1$ under the induction hypothesis.
The coefficient is also $C^1$ and has a positive limit.  The scalar estimate
shows that $Y$ is $C^1$ at zero.  Hence $A\in C^{\ell+1}$, after which the
regular equation gives $Z\in C^{\ell+2}$.  Induction proves the claimed
$C^\infty$ extension.
\end{proof}

\section{The general dimension bound and conditional strictness}
\label{sec:general}
The proof uses all the inequalities defining \(\Gamma_k^+\), rather than
only the equation for \(\sigma_k\).  Indeed, under the assumption
\(p>p_k(n)\),  \cref{lem:model}  shows that for every relevant value of \(\alpha\)
there is an index \(j\in\{1,\ldots,k\}\) such that
\[
R_j(\alpha)\leq-\delta_R.
\]
The radial block decomposition in \cref{lem:acceleration} then converts this inequality,
together with \(\sigma_j>0\), into the uniform estimate
\[
\eta''\geq h>0.
\]

\subsection{Proof of the model-threshold bound}

\begin{proof}[Proof of \cref{thm:main}]
Suppose $p>p_k(n)$.  By \eqref{eq:old-use},
$\operatorname{codim}\Sigma\geq2$, so part~(ii) of
\cref{lem:peak-preparations} applies.  Put $t=-\log r$ and
$f=\log(v/r)$.  By \cref{lem:scale-gradient}, $f$ is uniformly Lipschitz in
$t$, with constant $L_0$.  Set $H_*:=L_0+4$ and fix the constants in
\cref{lem:acceleration} corresponding to this value.  Choose
$c\in(0,\bar a)$ and then $T$ so large that
\begin{equation}
\Delta:=cT-\frac{(L_0+c)^2}{2h}>0.
\label{eq:Delta}
\end{equation}

Let $r_*$ and $C_*$ be the constants in part~(ii) of
\cref{lem:peak-preparations}.  Choose $K>2C_*$ and then
$\Lambda$ so large that, with
\[
\mu=\frac{Ke^T}{\Lambda},
\]
one has $C_*\mu<\varepsilon_0$.
Take $t_0$ so large that all subsequent boxes lie inside the tubular
  neighborhood, satisfy $\Lambda\rho_j<\operatorname{inj}(\Sigma)/4$, and have
$r<\min\{r_0,r_*\}$.  Set
\[
t_j=t_0+jT,
\qquad
  \rho_j=e^{-t_j}.
\]
  Starting from $X_0=(y_0,\rho_0,\theta_0)$, define recursively
\[
  P_j(y)=\mu\frac{d_\Sigma(y,y_j)^2}{\rho_j^2}
\]
and
\[
\eta_j(t)=
  \max_{\substack{d_\Sigma(y,y_j)\leq\Lambda\rho_j\\
\theta\in\Sn(N_y\Sigma)}}
\{f(t,y,\theta)-P_j(y)\}.
\]
Since every member of the maximizing family is $L_0$-Lipschitz in $t$,
$\eta_j$ has the same Lipschitz constant on $[t_j,t_{j+1}]$.
  At a boundary point of the base ball,
  \[
  P_j(y)=\mu\Lambda^2=Ke^T\Lambda.
  \]
  For every $t\in[t_j,t_{j+1}]$, \eqref{eq:box-osc} gives
  \[
  f(t,y,\theta)-P_j(y)
  \leq
  \sup_{\vartheta\in\Sn(N_{y_j}\Sigma)}f(t,y_j,\vartheta)
  +C_*(\Lambda e^T+T+1)-Ke^T\Lambda.
  \]
  Since $K>2C_*$, after enlarging $\Lambda$ we have
  \[
  Ke^T\Lambda>C_*(\Lambda e^T+T+1).
  \]
  Hence
  \[
  \sup_{\substack{d_\Sigma(y,y_j)=\Lambda\rho_j\\
  \theta\in\Sn(N_y\Sigma)}}
  \{f(t,y,\theta)-P_j(y)\}
  <
  \max_{\vartheta\in\Sn(N_{y_j}\Sigma)}f(t,y_j,\vartheta).
  \]
 Since \(P_j(y_j)=0\),
\[
\max_{\vartheta\in\mathbb S(N_{y_j}\Sigma)}
f(t,y_j,\vartheta)
=
\max_{\vartheta\in\mathbb S(N_{y_j}\Sigma)}
\bigl\{f(t,y_j,\vartheta)-P_j(y_j)\bigr\}.
\]
Thus the penalized function
\[
(y,\theta)\longmapsto f(t,y,\theta)-P_j(y)
\]
takes at an interior point with \(y=y_j\) a value strictly larger than
all its values on the boundary of the base ball.  Hence every maximizer
satisfies
\[
d_\Sigma(y,y_j)<\Lambda\rho_j.
\]
Moreover,
\[
\eta_j(t_j)\geq f(X_j).
\]

Suppose that $\phi\in C^2$ touches $\eta_j$ from above at an interior time
$t$, and choose $(y,\theta)$ attaining the maximum in the definition of
$\eta_j(t)$.  For nearby $\tau$ and every admissible
$(\widetilde y,\widetilde\theta)$,
\[
f(\tau,\widetilde y,\widetilde\theta)-P_j(\widetilde y)
\leq\eta_j(\tau)\leq\phi(\tau),
\]
with equality at $(t,y,\theta)$.  Therefore
\[
(\tau,\widetilde y,\widetilde\theta)
\longmapsto\phi(\tau)+P_j(\widetilde y)
\]
touches $f$ from above at $(t,y,\theta)$, and
\[
\partial_tf(t,y,\theta)=\phi'(t),
\qquad
D_\Sigma f(t,y,\theta)=D_\Sigma P_j(y),
\qquad
D_\theta f(t,y,\theta)=0.
\]
Since $\eta_j$ is $L_0$-Lipschitz and
\eqref{eq:scale-gradient} holds, these identities give
\[
|\phi'(t)|+r|D_\Sigma P_j(y)|\leq H_*.
\]
The function
\[
r\exp\{\phi(-\log r)+P_j(y)\}
\]
is consequently an upper test for $v$ satisfying the contact bounds in
\cref{lem:acceleration}.  Hessian ordering and monotonicity of the
G\aa rding cone put its matrix in $\Gk$, while
\eqref{eq:box-penalty} gives
\[
r^2|D_\Sigma^2P_j|\leq C_*\mu<\varepsilon_0.
\]

  For fixed $j$, let
  \[
  K_j=
  \bigl\{(y,\theta):d_\Sigma(y,y_j)\leq\Lambda\rho_j,
  \ \theta\in\Sn(N_y\Sigma)\bigr\}.
  \]
  This set is compact.  On a slightly larger tube box there is $C_j>0$ such
  that
  \[
  \partial_t^2f(t,y,\theta)\geq-C_j
  \]
  for $(y,\theta)\in K_j$.  Since $P_j$ is independent of $t$,
  \[
  \eta_j(t)+\frac{C_j}{2}t^2
  =\max_{(y,\theta)\in K_j}
  \left\{f(t,y,\theta)-P_j(y)+\frac{C_j}{2}t^2\right\}
  \]
  is convex.  Hence $\eta_j$ is locally semiconvex.  No bound on $C_j$
  uniform in $j$ is required in part~(i) of
  \cref{lem:peak-preparations}.  By \cref{lem:acceleration}, every interior
  upper test satisfies
\[
\phi'\leq\bar a\quad\Longrightarrow\quad\phi''\geq h.
\]
Part~(i) of \cref{lem:peak-preparations} and \eqref{eq:Delta} give
\[
\eta_j(t_{j+1})\geq\eta_j(t_j)+\Delta.
\]
Choose a maximizer $X_{j+1}$ at the terminal time.  Since $P_j\geq0$,
resetting the next penalty at its base point gives the unpenalized growth
\begin{equation}
f(X_{j+1})\geq f(X_j)+\Delta,
\qquad
f(X_j)\geq f(X_0)+j\Delta.
\label{eq:growth}
\end{equation}

The points $X_j,X_{j+1}$ satisfy the hypotheses of the path construction in
part~(ii) of \cref{lem:peak-preparations}.  Using \eqref{eq:growth} in
\eqref{eq:box-path} gives
\[
\operatorname{Length}_g(\gamma_j)
  \leq C\frac{\rho_j}{\rho_{j+1}e^{f(X_{j+1})}}
\leq C_0e^{-j\Delta},
\]
where $C_0$ is independent of $j$ and absorbs
$T,\Lambda$, and $f(X_0)$.
The concatenated curve has finite $g$-length.  Moreover, the construction of
the maximizers gives
\[
  d_\Sigma(y_j,y_{j+1})\leq\Lambda\rho_j.
\]
  Since $\sum_j\rho_j<\infty$, the sequence $(y_j)$ is Cauchy in the compact
manifold $\Sigma$ and hence converges to some $y_\infty\in\Sigma$.  Also
  $\rho_j\to0$.  More precisely, for every $Z\in\gamma_j$,
\[
\operatorname{dist}_{\rm Eucl}(Z,\Sigma)
\leq
  \operatorname{dist}_{\rm Eucl}(Z,X_j)+\rho_j
\leq
  \operatorname{Length}_{\rm Eucl}(\gamma_j)+\rho_j
\leq
  \bigl(1+C_*(\Lambda+1)\bigr)\rho_j.
\]
Thus the entire tail of the concatenated curve approaches $\Sigma$ and
leaves every compact subset of $\Sn^n\setminus\Sigma$.  It has finite
$g$-length but no limit in $\Sn^n\setminus\Sigma$, contradicting metric
completeness.  Hence $p\leq p_k(n)$.
\end{proof}

\subsection{Strictness under finite linear contact}

The equality examples rule out a universal strict theorem but do not prevent
strictness under additional asymptotic control.

\begin{theorem}
\label{thm:strict-contact}
In addition to the hypotheses of \cref{thm:main}, assume
\[
0<L_*:=\limsup_{x\to\Sigma}
\frac{v(x)}{\operatorname{dist}(x,\Sigma)}<\infty.
\]
Then
\[
p<p_k(n).
\]
\end{theorem}

\begin{proof}
Let $\Pi:\mathbb R^n\to\Sn^n$ be the stereographic chart used above and set
$\Sigma_{\rm E}:=\Pi^{-1}(\Sigma)$.  Write
\[
\Pi^*g_{\Sn^n}=\omega^2|dx|^2.
\]
If $v_{\Sn}$ and $v_{\rm E}$ denote the spherical and Euclidean conformal
factors, respectively, then
\[
v_{\Sn}\circ\Pi=\omega v_{\rm E}.
\]
For $x\to\Sigma_{\rm E}$, let $y_x\in\Sigma_{\rm E}$ be a Euclidean nearest
point.  Tubular-coordinate expansion gives, uniformly as
$x\to\Sigma_{\rm E}$,
\[
\operatorname{dist}_{\Sn^n}(\Pi(x),\Sigma)
=\omega(y_x)\operatorname{dist}_{\rm E}(x,\Sigma_{\rm E})
\bigl(1+o(1)\bigr).
\]
Consequently,
\[
\frac{v_{\Sn}(\Pi(x))}
{\operatorname{dist}_{\Sn^n}(\Pi(x),\Sigma)}
=\frac{\omega(x)}{\omega(y_x)}
\frac{v_{\rm E}(x)}{\operatorname{dist}_{\rm E}(x,\Sigma_{\rm E})}
\bigl(1+o(1)\bigr)
=\frac{v_{\rm E}(x)}{\operatorname{dist}_{\rm E}(x,\Sigma_{\rm E})}
\bigl(1+o(1)\bigr).
\]
Thus the finite positive contact condition is invariant under the change of
gauge.  We henceforth write $v$ and $\Sigma$ for the Euclidean conformal
factor and singular submanifold, and retain the notation $L_*$ for the
limsup.  Choose $x_j\to\Sigma$ such that
\[
r_j:=\operatorname{dist}_{\rm E}(x_j,\Sigma)\longrightarrow0,
\qquad
\frac{v(x_j)}{r_j}\longrightarrow L_*,
\]
and let $y_j\in\Sigma$ be a Euclidean nearest point to $x_j$.  Set
\[
e_j:=\frac{x_j-y_j}{r_j}\in N_{y_j}\Sigma.
\]
Fix a unit vector $e_0\in\{0\}\times\mathbb R^{n-p}$ and choose
$Q_j\in O(n)$ such that
\[
Q_j(\mathbb R^p\times\{0\})=T_{y_j}\Sigma,
\qquad
Q_j(\{0\}\times\mathbb R^{n-p})=N_{y_j}\Sigma,
\qquad
Q_je_0=e_j.
\]
Define
\[
v_j(X)=r_j^{-1}v(y_j+r_jQ_jX),
\qquad
\Sigma_j=r_j^{-1}Q_j^{-1}(\Sigma-y_j).
\]
After taking a subsequence, the smooth rescaled submanifolds converge on
compact sets to $\mathbb R^p\times\{0\}$, and
$\operatorname{dist}(X,\Sigma_j)\to|Z|$ locally uniformly, where
$X=(Y,Z)$.  The definition of the global limsup means that, for every
$\epsilon>0$, the inequality
\[
v(x)\leq(L_*+\epsilon)\operatorname{dist}(x,\Sigma)
\]
holds throughout a sufficiently small fixed neighborhood of $\Sigma$.
Consequently
\[
v_j(X)\leq(L_*+\epsilon)(|Z|+o(1))
\]
locally uniformly.  On the other hand,
$v_j(0,e_0)=v(x_j)/r_j\to L_*$.  Thus the upper half-relaxed limit $U$
satisfies
\[
U\leq L_*|Z|,
\qquad
U(0,e_0)=L_*,
\]
and is touched from above at $X_0=(0,e_0)$ by the smooth local test
\[
\Phi=L_*|Z|.
\]

We first prove that the upper test $\Phi$ satisfies
\[
\mathcal A[\Phi](X_0)\in\overline{\Gk},
\qquad
\sigma_k(\mathcal A[\Phi](X_0))\geq\kappa.
\]
Fix a ball $B_\rho(X_0)$ disjoint from $\{Z=0\}$ and put
$\Phi_\delta=\Phi+\delta|X-X_0|^2$.  The contact with $U$ is strict on the
boundary of this ball.  The definition of the upper half-relaxed limit then
provides local maxima $X_{j,\delta}\to X_0$ of
$v_j-\Phi_\delta$, after passage to a subsequence, with
\[
c_{j,\delta}:=v_j(X_{j,\delta})-
\Phi_\delta(X_{j,\delta})\longrightarrow0.
\]
The function $\Phi_\delta+c_{j,\delta}$ touches $v_j$ from above at
$X_{j,\delta}$.  Affine scaling and rotation preserve the Euclidean formula
\eqref{eq:A-v}, so
\[
\mathcal A[v_j](X)
=Q_j^T\mathcal A[v](y_j+r_jQ_jX)Q_j,
\qquad
\sigma_k(\mathcal A[v_j])=\kappa.
\]
At the contact point, the two functions have the same positive value and the
same gradient, while the test has the larger Hessian.  Therefore
\[
\mathcal A[\Phi_\delta+c_{j,\delta}]
-\mathcal A[v_j]
=v_j(X_{j,\delta})
\left(D^2(\Phi_\delta+c_{j,\delta})(X_{j,\delta})
-D^2v_j(X_{j,\delta})\right)\geq0.
\]
Monotonicity of $\Gamma_k^+$ and of $\sigma_k$ on this cone therefore gives
\[
\mathcal A[\Phi_\delta+c_{j,\delta}]\in\Gk,
\qquad
\sigma_k(\mathcal A[\Phi_\delta+c_{j,\delta}])\geq\kappa.
\]
First let $j\to\infty$ with $\delta>0$ fixed.  Since
$X_{j,\delta}\to X_0$ and $c_{j,\delta}\to0$, continuity gives
\[
\mathcal A[\Phi_\delta](X_0)\in\overline{\Gk},
\qquad
\sigma_k(\mathcal A[\Phi_\delta](X_0))\geq\kappa.
\]
Letting $\delta\downarrow0$, we obtain
\[
\mathcal A[\Phi](X_0)\in\overline{\Gk},
\qquad
\sigma_k(\mathcal A[\Phi](X_0))\geq\kappa.
\]
The eigenvalues of $\mathcal A[\Phi](X_0)$ are $L_*^2/2$ with multiplicity
$n-p-1$ and $-L_*^2/2$ with multiplicity $p+1$.  Thus
\[
\sigma_j(\mathcal A[\Phi](X_0))
=\left(\frac{L_*^2}{2}\right)^jc_{n,p,j}
\]
for $1\leq j\leq k$, and in particular
\[
\left(\frac{L_*^2}{2}\right)^kc_{n,p,k}\geq\kappa>0.
\]
The main theorem gives $p\leq p_k(n)$.  Equality would imply
$c_{n,p,k}=0$ by \cref{lem:model}, a contradiction.
\end{proof}

\section[The quadratic case and optimality]
{The quadratic case and optimality}\label{sec:k2}

For $k=2$, the model threshold and the upper-contact remainder are quadratic.
We first display the resulting simplification of the general proof and then
construct complete metrics attaining the critical dimension.

\subsection{The quadratic specialization}

For comparison, the algebraic part of the dimension estimate is elementary
when $k=2$.  Put $d=n-2p-2$.  The recurrence
\eqref{eq:recurrence} gives
\[
\mathcal P_{2,n}(d)=d^2-n.
\]
Therefore
\begin{equation}
p_2(n)=\frac{n-2-\sqrt n}{2}.
\label{eq:p2}
\end{equation}
If $p>p_2(n)$, then $0<d<\sqrt n$, and the two model functions are
\begin{align*}
R_1(\alpha)&=d+(n-2)\alpha,\\
R_2(\alpha)&=\frac{d^2-n}{2}
+\bigl(d(n-3)-2\bigr)\alpha
+\frac{(n-1)(n-4)}{2}\alpha^2.
\end{align*}
Thus $R_1(0)>0$ and $R_2(0)<0$.  The cone-monotonicity argument in
\cref{lem:model}, together with $R_1(\alpha)\leq-2p$ for
$\alpha\leq-1$, gives constants $\bar a\in(0,1)$ and $\delta_2>0$ such that
\[
\min\{R_1(\alpha),R_2(\alpha)\}\leq-\delta_2
\qquad\text{for every }\alpha\leq\bar a.
\]
For an upper test \eqref{eq:test}, set
\[
\chi=\frac{r^2|D_\Sigma P|^2}{1-\eta'},
\qquad
\alpha=\eta'-\chi.
\]
Then \eqref{eq:Gidentity} becomes
\[
G_1=(1-\eta')R_1(\alpha),
\qquad
G_2=(1-\eta')^2\bigl(R_2(\alpha)-4\chi\bigr).
\]
This is the explicit $k=2$ version of the separation used in
\cref{lem:acceleration}; it yields the same implication
\[
\eta'\leq\bar a
\quad\Longrightarrow\quad
\eta''\geq h_2
\]
under the contact hypotheses of that lemma.  The argument in the proof of
\cref{thm:main} then applies without change and excludes $p>p_2(n)$. 

\subsection{Equality at the critical dimension}

The critical value in \eqref{eq:p2} is an integer precisely when
$n=m^2$ for an integer $m\geq3$.  In that case
\[
p=p_2(n)=\frac{m^2-m-2}{2}.
\]
We construct a complete metric at this critical value.

\begin{proof}[Proof of \cref{thm:sharp}]
Put
\[
q=\frac{m(m+1)}2,
\qquad
P_-=p+1=\frac{m(m-1)}2.
\]
In tubular coordinates about an equatorial $\Sn^p\subset\Sn^n$,
\[
g_{\Sn^n}=d\rho^2+\cos^2\rho\,g_{\Sn^p}
+\sin^2\rho\,g_{\Sn^q}.
\]
With $\sinh s=\cot\rho$,
\begin{equation}
g_0:=\sin^{-2}\rho\,g_{\Sn^n}
=ds^2+\sinh^2s\,g_{\Sn^p}+g_{\Sn^q},
\label{eq:g0}
\end{equation}
so $(\Sn^n\setminus\Sn^p,g_0)$ is
$\Hh^{p+1}(-1)\times\Sn^q(1)$.  The Schouten eigenvalues of $g_0$ are
$-1/2$ with multiplicity $P_-$ and $1/2$ with multiplicity $q$.  Since
$q-P_-=m$ and $P_-+q=m^2$,
\[
\sigma_1(g_0^{-1}A_{g_0})=\frac m2>0,
\qquad
\sigma_2(g_0^{-1}A_{g_0})=0.
\]

Seek $g=e^{-2\eta(s)}g_0$, and set
\[
a=\eta',
\qquad
z=e^{-4\eta}.
\]
The eigenvalue blocks of $2e^{-2\eta}g^{-1}A_g$ are
\[
\mu_0=-1+2a'+a^2,
\qquad
\mu_H=-1+2a\coth s-a^2\quad[p],
\qquad
\mu_S=1-a^2\quad[q].
\]
For $u=a\coth s$, define
\begin{align*}
C(a,u)&=p\mu_H+q\mu_S,\\
B(a,u)&=\binom p2\mu_H^2+pq\mu_H\mu_S
+\binom q2\mu_S^2,\\
G(a,u)&=B(a,u)+(a^2-1)C(a,u).
\end{align*}
The constant equation is equivalent to
\begin{align}
a'&=\frac{4\kappa z-G(a,a\coth s)}
{2C(a,a\coth s)},
\label{eq:sigma2-system-a}\\
z'&=-4az.
\label{eq:sigma2-system-z}
\end{align}
At $(a,u)=(0,0)$,
\[
C=m+1,
\qquad
G=0,
\qquad
\partial_uG=L_c:=(m-2)(m+1)^2=2p(m+1).
\]
In a fixed neighborhood on which $C>0$, define the exact remainder
\[
R(a,u,z):=
\frac{4\kappa z-G(a,u)}{2C(a,u)}
+pu-\frac{2\kappa}{m+1}z.
\]
The functions $C$ and $G$ are even in $a$, and direct differentiation at
the origin gives
\[
R(0,0,0)=0,
\qquad
DR(0,0,0)=0.
\]
Moreover, $R$ is affine in $z$.  Taylor's theorem therefore gives, after
possibly shrinking the neighborhood,
\[
|R(a,u,z)|\leq K\bigl(u^2+z|u|\bigr)
\]
whenever $|a|\leq|u|$ and $z\geq0$.  Equation~\eqref{eq:sigma2-system-a}
can consequently be written as
\begin{equation}
a'+p\coth s\,a
=\frac{2\kappa}{m+1}z+R,
\label{eq:sigma2-stable}
\end{equation}

We construct the solution from the smooth orbit $s=0$.  Smooth radiality
requires $a(s)=cs+O(s^3)$.  At the origin the radial direction joins the
$p$ hyperbolic angular directions, and the derivative of the $\sigma_2$
equation with respect to $c$ is $2P_-(m+1)>0$.  The implicit-function theorem
gives the positive branch
\begin{equation}
c=\frac{2\kappa}{P_-(m+1)}\varepsilon+O(\varepsilon^2)
\label{eq:sigma2-c}
\end{equation}
for $z(0)=\varepsilon>0$.

Equation~\eqref{eq:sigma2-stable}, together with
\eqref{eq:sigma2-system-z}, is equivalent to the Volterra system
\begin{align}
a(s)&=\sinh^{-p}s\int_0^s\sinh^p\tau
\left(\frac{2\kappa}{m+1}z(\tau)+R(\tau)\right)d\tau,
\label{eq:sigma2-volterra-a}\\
z(s)&=\varepsilon\exp\left(-4\int_0^sa(\tau)d\tau\right).
\label{eq:sigma2-volterra-z}
\end{align}
Put
\[
A_0=\frac{2\kappa}{(m+1)P_-}.
\]
On $[0,s_0]$ use the affine Banach space
\[
\mathcal X_\varepsilon:=\left\{(a,z)\in C([0,s_0])^2:
a(0)=0,\ z(0)=\varepsilon,\
\sup_{0<s\leq s_0}\frac{|a(s)|}{s}
+\sup_{0<s\leq s_0}\frac{|z(s)-\varepsilon|}{s^2}<\infty\right\}
\]
with metric
\[
d_*\bigl((a_1,z_1),(a_2,z_2)\bigr)
:=\sup_{0<s\leq s_0}\frac{|a_1(s)-a_2(s)|}{s}
+\sup_{0<s\leq s_0}\frac{|z_1(s)-z_2(s)|}{s^2}.
\]
This is complete.  Consider its closed subset
\[
\mathcal B=\left\{(a,z):
 \sup_{0<s\leq s_0}\frac{|a(s)|}{s}\leq M\varepsilon,
 \quad
 \sup_{0<s\leq s_0}\frac{|z(s)-\varepsilon|}{s^2}
 \leq3M\varepsilon^2\right\},
\]
where
\[
M>\max\left\{2A_0,\frac{4\kappa}{p(m+1)}\right\}
\]
is fixed once and for all.  Decrease $s_0$ and $\varepsilon$ so that $z>0$
on this set.  Here $u=a\coth s=O(\varepsilon)$, and the explicit formula for
$R$ and the vanishing of its first derivatives at the origin give
\begin{align*}
|R|&\leq C\varepsilon^2,\\
|R(a_1,u_1,z_1)-R(a_2,u_2,z_2)|
&\leq C\varepsilon
\bigl(|(a_1-a_2)\coth s|+|z_1-z_2|\bigr).
\end{align*}
The elementary estimates
\[
\sinh^{-p}s\int_0^s\sinh^p\tau\,d\tau
=\frac{s}{P_-}+O(s^3),
\qquad
\int_0^s\tau\,d\tau=\frac{s^2}{2}
\]
show that the right sides of
\eqref{eq:sigma2-volterra-a}--\eqref{eq:sigma2-volterra-z} define a map
$\mathcal T$.  If $(\widetilde a,\widetilde z)=\mathcal T(a,z)$, then
\[
\sup_{0<s\leq s_0}\frac{|\widetilde a(s)|}{s}
\leq A_0\varepsilon+C\varepsilon^2+C\varepsilon s_0^2
<M\varepsilon.
\]
Choose $s_0$ first and then $\varepsilon$.  Moreover,
\[
\left|\int_0^s a(\tau)\,d\tau\right|
\leq\frac{M\varepsilon s^2}{2},
\]
and hence
\[
\frac{|\widetilde z(s)-\varepsilon|}{s^2}
\leq2M e^{2M\varepsilon s_0^2}\varepsilon^2
\leq3M\varepsilon^2.
\]
Thus $\mathcal T(\mathcal B)\subset\mathcal B$.  For two pairs in
$\mathcal B$, the same estimates, divided by the weights $s$ and $s^2$,
give
\[
d_*\bigl(\mathcal T(a_1,z_1),\mathcal T(a_2,z_2)\bigr)
\leq C(\varepsilon+s_0^2)
d_*\bigl((a_1,z_1),(a_2,z_2)\bigr).
\]
Decreasing the two parameters further makes $C(\varepsilon+s_0^2)<1$.
Thus $\mathcal T$ is a contraction.  Its fixed point satisfies
\[
a(s)=A_0\varepsilon s+O(\varepsilon^2s+\varepsilon s^3),
\qquad
z(s)=\varepsilon+O(\varepsilon^2s^2).
\]
It remains to justify that this continuous fixed point has the required
smoothness at the collapsed orbit.  Put $U(s)=a(s)/s$ for $s>0$.  Since
$a=sU$, $u=a\coth s=U(s)s\coth s$, the exact remainder can be written
\[
R(a(s),u(s),z(s))=\widehat R(s^2,U(s),z(s)),
\]
where $\widehat R$ is smooth.  Indeed, $s\coth s$ is smooth and even, while
$R$ is even in its first argument.  The vanishing of $DR$ at the origin also
gives $\partial_U\widehat R=O(\varepsilon)$ in the small box above.  Dividing the
  equation~\eqref{eq:sigma2-stable} by the radial scale gives
\[
sU'+(1+p\,s\coth s)U
=\frac{2\kappa}{m+1}z+\widehat R(s^2,U,z).
\]
The compatibility equation
\[
P_-c=\frac{2\kappa}{m+1}\varepsilon+\widehat R(0,c,\varepsilon)
\]
is precisely the implicit equation used in \eqref{eq:sigma2-c}; it has a
unique small root $c=A_0\varepsilon+O(\varepsilon^2)$.  Subtracting this
equation and applying the mean-value theorem in $U$ gives
\[
s(U-c)'+\lambda(s)(U-c)=g(s),
\qquad
\lambda(s)=P_-+O(\varepsilon+s^2),
\qquad
g(s)=O(s^2).
\]
Here the estimate for $g$ uses $z(s)-\varepsilon=O(\varepsilon^2s^2)$,
$s\coth s-1=O(s^2)$, and the smoothness of $\widehat R$.  Shrinking
$\varepsilon$ and $s_0$ gives $\lambda\geq P_-/2$.
With $\tau=-\log s$ and $W(\tau)=U(e^{-\tau})-c$, put
$\widehat\lambda(\tau)=\lambda(e^{-\tau})$ and
$\widehat g(\tau)=g(e^{-\tau})$.  Boundedness eliminates the growing
homogeneous mode and yields the variation-of-constants formula
\[
W(\tau)=\int_\tau^\infty
\exp\left(-\int_\tau^\xi\widehat\lambda(\rho)\,d\rho\right)
\widehat g(\xi)\,d\xi.
\]
Consequently $U(s)-c=O(s^2)$ as $s\downarrow0$.

We may therefore define continuous functions $A,Z$ at $x=0$ by
\[
a(s)=sA(s^2),
\qquad
z(s)=Z(s^2),
\qquad
A(0)=c,
\qquad
Z(0)=\varepsilon.
\]
Under $x=s^2$, equations~\eqref{eq:sigma2-system-a} and
\eqref{eq:sigma2-system-z} become
\begin{align}
A+2xA_x&=\mathcal F(x,A,Z),
\label{eq:sigma2-regular-A}\\
Z_x&=-2AZ,
\label{eq:sigma2-regular-Z}
\end{align}
where, with the smooth convention $\psi(0)=1$,
\[
\psi(x)=\sqrt{x}\coth\sqrt{x},
\qquad
\mathcal F(x,A,Z)=
\frac{2\kappa}{m+1}Z+\widehat R(x,A,Z)-p\psi(x)A.
\]
Thus $\mathcal F$ is smooth and
\[
\partial_A\mathcal F(0,c,\varepsilon)=-p+O(\varepsilon).
\]
The preceding estimate gives $A-c=O(x)$.  The compatibility relation is
$c=\mathcal F(0,c,\varepsilon)$, and, after decreasing $\varepsilon$, for
every $\ell\geq1$ one has
\[
2\ell+1-\partial_A\mathcal F(0,c,\varepsilon)
=2\ell+1+p+O(\varepsilon)>0.
\]
Hence \cref{lem:regular-singular}, with
$\mathcal G(x,A,Z)=-2AZ$, proves that $A$ and $Z$ extend smoothly to zero.
Consequently $s\mapsto sA(s^2)$ and $s\mapsto Z(s^2)$ have smooth odd and
even extensions, respectively.  Thus the solution is smooth across $s=0$,
and its initial slope is the branch \eqref{eq:sigma2-c}.

We turn to global existence.  The initial expansion and
\eqref{eq:sigma2-c} give $a(s)>0$ for small $s>0$.  If $s_1>0$ were its first
  zero, then the left derivative at $s_1$ would be nonpositive, whereas
  \eqref{eq:sigma2-system-a}, evaluated at $a=u=0$, would give
$a'(s_1)=2\kappa z(s_1)/(m+1)>0$.  Therefore, throughout the maximal
existence interval,
\[
a>0\quad(s>0),
\qquad
0<z\leq\varepsilon.
\]

Extend $u=a\coth s$ continuously by $u(0)=c$, and let $S_{\max}$ be the
supremum of the numbers $S$ such that the solution exists on $[0,S]$ and
$0\leq u\leq M\varepsilon$ there.  The local construction makes this set
nonempty.  On $[0,S_{\max})$, the exact remainder bound gives
\[
a'+(p-K_1\varepsilon)\coth s\,a
\leq\frac{2\kappa}{m+1}\varepsilon.
\]
Put $\beta_0=p-K_1\varepsilon\geq p/2$ and
\[
\overline a(s)=\frac{2\kappa}{\beta_0(m+1)}
\varepsilon\tanh s.
\]
This is a supersolution because
\[
\overline a'+\beta_0\coth s\,\overline a
\geq\frac{2\kappa}{m+1}\varepsilon.
\]
Moreover, \eqref{eq:sigma2-c} gives
\[
c<\frac{2\kappa}{\beta_0(m+1)}\varepsilon
\]
for sufficiently small $\varepsilon$, so comparison after multiplication by
$\sinh^{\beta_0}s$ yields
\[
0\leq a(s)\leq\overline a(s),
\qquad
0\leq u(s)\leq\frac{2\kappa}{\beta_0(m+1)}\varepsilon.
\]
Since $\beta_0\to p$, after shrinking $\varepsilon$ we have
\[
\frac{2\kappa}{\beta_0(m+1)}
<\frac{4\kappa}{p(m+1)}<M.
\]
Thus $u\leq c_0\varepsilon$ on $[0,S_{\max})$ for a constant $c_0<M$.
On the region $0\leq u\leq c_0\varepsilon$,
\[
|C-(m+1)|\leq C\varepsilon,
\qquad
C\geq\frac{m+1}{2}.
\]
Consequently, the vector field defined by
\eqref{eq:sigma2-system-a}--\eqref{eq:sigma2-system-z} is smooth and uniformly
bounded on the relevant compact set of $(a,a\coth s,z)$-values.  If
$S_{\max}<\infty$, the standard continuation theorem for ordinary
differential equations, applied at the positive time $S_{\max}$, extends the
  solution beyond $S_{\max}$.  Since $c_0<M$, continuity also preserves
  $u\leq M\varepsilon$ on a larger interval, contradicting the definition of
  $S_{\max}$.
Therefore $S_{\max}=\infty$.  Thus the solution exists for every $s\geq0$,
and \eqref{eq:sigma2-system-a} also yields
\[
|a'|\leq C\varepsilon.
\]

The decreasing function $z$ has a limit $z_\infty\geq0$.  If this limit
were positive, then $z'=-4az$ would make $a$ integrable, whereas the lower
comparison
\[
a'+(p+K_2\varepsilon)\coth s\,a
\geq\frac{2\kappa}{m+1}z_\infty
\]
would give a positive lower bound for $a$ at large $s$.  Therefore
$z_\infty=0$.  For large $s$, the remainder estimate and the global small
bounds imply
\[
a'+\frac p2a\leq \frac{2\kappa}{m+1}z.
\]
Duhamel's formula, together with $z(s)\to0$, now gives $a(s)\to0$.

For large $s$,
\begin{equation}
a'+pa=\frac{2\kappa}{m+1}z+N,
\qquad
|N|\leq K(e^{-2s}a+a^2+za).
\label{eq:sigma2-asymptotic-stable}
\end{equation}
  For the comparison between the stable variable $a$ and the center variable
  $z$, choose $\delta>0$ and $\beta\in(0,p)$ so that
  \[
  4\delta<\beta,
  \qquad
  3K\delta<p-\beta.
  \]
  Since $a,z\to0$, after increasing $s_1$ we have
  \[
  e^{-2s}+a(s)+z(s)\leq\delta
  \]
  for $s\geq s_1$.  Equation~\eqref{eq:sigma2-asymptotic-stable} then
  implies, after absorbing the terms linear in $a$,
  \[
  a'+\beta a\leq C_1z.
  \]
  For $s_1\leq\tau\leq s$, the identity $z'/z=-4a$ and $0\leq a\leq\delta$
  give
  \[
  z(\tau)\leq z(s)e^{4\delta(s-\tau)}.
  \]
  Duhamel's formula therefore yields
  \[
  a(s)\leq e^{-\beta(s-s_1)}a(s_1)
  +C_1z(s)\int_{s_1}^s
  e^{-(\beta-4\delta)(s-\tau)}\,d\tau.
  \]
  Moreover,
  \[
  z(s)\geq z(s_1)e^{-4\delta(s-s_1)},
  \]
  so the first term is $o(z(s))$, while the integral is uniformly bounded.
  Consequently
  \[
  a=O(z).
  \]

  Set $y=a/z$.  Dividing \eqref{eq:sigma2-stable} by $z$, and using
  $z'/z=-4a$, $a=O(z)$, and the bound for the remainder, gives
\[
y'=\frac{2\kappa}{m+1}-py+o(1).
\]
  The function $y$ is bounded, and the integrating-factor formula for this
  scalar stable equation gives
\[
\frac az\longrightarrow\frac{2\kappa}{p(m+1)}
=\frac{4\kappa}{L_c}.
\]
Since $(1/z)'=4a/z$, it follows that
\[
\left(\frac1z\right)'=\frac{16\kappa}{L_c}+o(1).
\]
After integration,
\[
\frac1{z(s)}=\frac{16\kappa}{L_c}s+o(s),
\qquad
z(s)=\frac{L_c}{16\kappa s}(1+o(1)).
\]

The global estimates
\[
|a|+|a\coth s|+|a'|+z\leq C\varepsilon
\]
hold uniformly up to the smooth orbit $s=0$.  Hence the radial eigenvalue and
the $P_--1$ hyperbolic angular eigenvalues are $-1+O(\varepsilon)$, whereas
the $q$ spherical eigenvalues are $1+O(\varepsilon)$.  Consequently,
\[
\sigma_1=q-P_-+O(\varepsilon)=m+O(\varepsilon).
\]
After decreasing $\varepsilon$, we have
\[
\sigma_1\geq\frac m2>0.
\]
The defining curvature equation gives
\[
\sigma_2=4\kappa z>0
\]
for the normalized matrix.  Since
\[
\Gamma_2^+=\{\lambda:\sigma_1(\lambda)>0,\ \sigma_2(\lambda)>0\},
\]
it follows that $\lambda(g^{-1}A_g)\in\Gtwo$ everywhere.

The solution is smooth at $s=0$ by the preceding regularity argument.  At infinity,
\[
e^{-\eta(s)}=z(s)^{1/4}
=\left(\frac{L_c}{16\kappa s}\right)^{1/4}(1+o(1)).
\]
If
\[
F(S)=\int_0^S e^{-\eta(r)}dr,
\]
then $F(S)\to\infty$.  Every escaping curve has an unbounded $s$
subsequence because bounded hyperbolic balls and $\Sn^q$ are compact, and
if $s(t_j)\to\infty$, then
\[
\operatorname{Var}(F\circ s)
\geq|F(s(t_j))-F(s(t_0))|\longrightarrow\infty.
\]
Moreover,
\[
\operatorname{Length}_g(\gamma)
\geq\int e^{-\eta(s)}|ds|
=\operatorname{Var}(F\circ s).
\]
Thus $g$ is complete.  Finally, \eqref{eq:g0} gives
\[
g=V(\rho)^{-2}g_{\Sn^n},
\qquad
V(\rho)=\sin\rho\,e^{\eta(s)}.
\]
Since $s=\log(2/\rho)+o(1)$,
\[
V(\rho)=\left(\frac{16\kappa}{L_c}\right)^{1/4}
\rho\left(\log\frac1\rho\right)^{1/4}(1+o(1)).
\]
The metric is therefore a smooth global conformal metric on
$\Sn^n\setminus\Sn^p$, with a complete logarithmic end at $\Sn^p$.
\end{proof}

The equality examples include
\[
(n,p)=(9,2),(16,5),(25,9),(36,14),\ldots.
\]

\begin{remark}[Optimality of the dimension bound]
An analogous cohomogeneity-one construction gives equality at every integral
model endpoint.  More precisely, suppose
\[
p:=p_k(n)\in\mathbb Z
\]
and let $\Sigma=\mathbb S^p\subset\mathbb S^n$ be the totally geodesic
equatorial sphere
\[
\Sigma=
\bigl\{(x,0)\in\mathbb R^{p+1}\times\mathbb R^{n-p}:|x|=1\bigr\}.
\]
For every $\kappa>0$, the same type of construction as for $k=2$ yields a
smooth complete conformal metric $g$ on
$\mathbb S^n\setminus\Sigma$ satisfying
\[
\lambda(g^{-1}A_g)\in\Gamma_k^+,
\qquad
\sigma_k(g^{-1}A_g)=\kappa.
\]
Thus, whenever $p_k(n)$ is an integer, the dimension bound
\[
p\leq p_k(n)
\]
is attained.  Consequently, under the hypotheses of \cref{thm:main}, this
bound is optimal and cannot in general be replaced by
\[
p<p_k(n).
\]
\end{remark}

\paragraph{Declaration of generative AI and AI-assisted technologies in the manuscript preparation process}

 During the preparation of this work, the authors used ChatGPT (OpenAI) to assist in converting their proof manuscripts into a preliminary LaTeX draft and to identify potential grammatical and logical errors in the manuscript. After using this tool, the authors reviewed and edited the content as needed and take full responsibility for the content of the published article.

\paragraph{Acknowledgments}

The first and third authors were supported by the National Natural Science Foundation of China [grant number: 2025YFA1017601].

\bibliographystyle{alpha}
\bibliography{reference}

\begingroup
\footnotesize
\noindent
(Jiahuan Li) School of Mathematical Sciences, University of Science and
Technology of China, Hefei, 230026, Anhui Province, P.R. China.\\
Email address: \href{mailto:jiahuan@mail.ustc.edu.cn}{jiahuan@mail.ustc.edu.cn}\par\smallskip
\noindent
(Yilu Liu) School of Mathematical Sciences, University of Science and
Technology of China, Hefei, 230026, Anhui Province, P.R. China.\\
Email address: \href{mailto:liuylgeoanaly@mail.ustc.edu.cn}{liuylgeoanaly@mail.ustc.edu.cn}\par\smallskip
\noindent
(Xi-Nan Ma) School of Mathematical Sciences, University of Science and
Technology of China, Hefei, 230026, Anhui Province, P.R. China.\\
Email address: \href{mailto:xinan@ustc.edu.cn}{xinan@ustc.edu.cn}
\endgroup

\end{document}